\documentclass{article}
\usepackage{maa-monthly}

\theoremstyle{theorem}
\newtheorem{theorem}{Theorem}
\newtheorem{lemma}[theorem]{Lemma}
\newtheorem{corollary}[theorem]{Corollary}
 
\theoremstyle{definition}
\newtheorem{definition}[theorem]{Definition}
\newtheorem{example}[theorem]{Example}

\begin{document}

\title{Asymptotics on a Class of Legendre Formulas}
\markright{Asymptotics on Legendre Formulas}
\author{Maiyu Diaz}

\maketitle

\begin{abstract}
Let $f$ be a real-valued function of a single variable such that it is positive over the primes. In this article, we construct a factorial, $n!_f$, associated to $f$, called the associated Legendre formula, or $f$-factorial, and show, subject to certain criteria, that $n!_f$ satisfies a weak Stirling approximation. As an application, we will give weak approximations to the Bhargava factorial over the set of primes and to a less well-known Legendre formula.
\end{abstract}

\noindent

\section{Introduction.} \label{sec1}

The factorial function is a fundamental object of utmost importance, finding itself in various disciplines, including combinatorics, number theory, ring theory, and many others. Places where it arises include, but are not limited to, being a building block for binomial coefficients, being used to count the number of permutations of a collection of objects, and contributions to a major part of Taylor series expansions. The study of obtaining asymptotics to $n!$ dates back to the 17th and 18th century, where, in particular, Abraham de Moivre showed that \cite{lecam}
$$n! \sim Ke^{-n}n^{n+1/2}$$
for some constant $K$. Later, in the 18th century, James Stirling \cite{pearson} was able to specify the constant: $K = \sqrt{2 \pi}$.

At the turn of the 21st century, Manjul Bhargava introduced a generalization of the factorial in \cite{bhargava}. 

\begin{definition} \label{def1}
Let $S\subseteq \mathbb{Z}$ be arbitrary and fix a prime $p$. Construct a sequence $\{a_i\}_{i \geq 0}$, called a \emph{$p$-ordering} of $S$, of elements of $S$ as follows:
\begin{itemize}
\item Choose any element $a_0 \in S$.
\item Choose an element $a_1 \in S$ that minimizes the highest power of $p$ dividing $a_1-a_0$.
\item In general, for each $n \ge 1$, choose an element $a_n \in S$ that minimizes the highest power of $p$ dividing $(a_n -a_0)(a_n-a_1)\cdots(a_n-a_{n-1})$.
\end{itemize}
If $w_p(a)$ denoes the highest power of $p$ dividing $a$, then define
$$v_n(S,p) = w_p((a_n-a_0)(a_n-a_1)\cdots(a_n-a_{n-1})).$$
Then the \textit{Bhargava factorial over} $S$, denoted $n!_S$, is defined as
$$n!_S = \prod_{p} v_n(S,p).$$
\end{definition}

The Bhargava factorial $n!_S$, with $S \subseteq \mathbb{Z}$, satisfies four properties analogous to the factorial:

\begin{enumerate}

\item{For any nonnegative integers $k$ and $l$, $(k+l)!_S$ is a multiple of $k!_S l!_S$.}

\item{Let $f$ be a primitive polynomial of degree $k$ and let $d(S, f) = \mathrm{gcd} \, \{f(a) \, | \, a \in S \}$. Then $d(S, f)$ divides $k!_S$.}

\item{Let $a_0, a_1, \dots, a_n \in S$ be any $n+1$ integers. Then the product $$\prod_{i < j} (a_i - a_j)$$ is a multiple of $0!_S 1!_S \cdots n!_S$.}

\item{The number of polynomial functions from $S$ to $\mathbb{Z}/n\mathbb{Z}$ is given by $$\prod_{k=0}^{n-1} \frac{n}{\mathrm{gcd}\, (n, k!_S)}.$$}

\end{enumerate}

Manjul Bhargava in \cite{bhargava} remarks the following:

\bigskip
\noindent\textbf{Question 31:} \text{ What are analogues of Stirling's formula for generalized factorials?}
\bigskip

\noindent In this article, we obtain weak asymptotic formulas for $\mathbb{P}$, the set of primes, and a set $S'$, which we will define momentarily. Specifically, we prove the following theorems.
\begin{theorem} \label{thm2}
Let $n!_\mathbb{P}$ denote the Bhargava factorial over the set of primes $\mathbb{P}$. Then
$$\log (n+1)!_\mathbb{P} = \log n! + Cn + o(n),$$
where $$C = \sum_{p \in \mathbb{P}} \frac{\log p}{(p-1)^2} = 1.2269688\dots .$$
\end{theorem}
Consider the sequence $$1, 6, 360, 45360, 5443200, 359251200, \dots$$ which can be found as sequence A202367 at \url{www.oeis.org}. While this sequence may not be obviously related to the Bhargava factorial, it is interesting to note that in \cite{nitesh}, N.~Mathur devised an algorithm to show that, for any sequence under specific criteria, we can determine a set $S \subseteq \mathbb{Z}$ associated with the sequence. From this, we can determine a set $S'$ associated with the sequence A202367; a quick computation of the first few terms of the set gives $S' = \{2,4,16,22,\dots\}$, outputting $n!_{S'} = \{1,1,6,360,\dots\}$.
\begin{theorem} \label{thm3}
Let $n!_{S'}$ denote the Bhargava factorial over the set associated with the sequence A202367. Then
$$\log (n+1)!_{S'} = \log (2n)! + \beta n + o(n),$$
where $$ \beta = \sum_{p \in \mathbb{P}} \frac{\log p}{(p-1)} \left( \frac{p}{\left \lceil (p-1)/2 \right \rceil} - 2\right) = 1.0676431\dots .$$
\end{theorem}
The proofs of Theorems \ref{thm2} and \ref{thm3} will be given in Section \ref{sec4}. Each of the Bhargava factorials discussed in \cite{bhargava} and in sequence A202367 yields a Legendre formula that can be expressed as
\begin{equation}
\label{ffactorial} \prod_{p \in \mathbb{P}} p^{\sum_{k \geq 0 } \left \lfloor\frac{n}{f(p)p^k} \right \rfloor},
\end{equation}
for some real-valued map $f$---as such, in determining the asymptotics for the Bhargava factorials discussed, we will investigate (\ref{ffactorial}).
In Section \ref{sec2}, we will define the associated Legendre formula for a real-valued map, or alternatively for the $f$-factorial, and prove specific properties about them and make connections with the Bhargava factorial.

In Section \ref{sec3}, we will prove, under certain criteria for $f$, an asymptotic formula for the $f$-factorial---a general result that will provide approximations for the Bhargava factorial over $\mathbb{P}$ and $S'$. Due to the lack of assumptions on $f$, we consider classes of functions $f$ that can be approximated ``well-enough" by linear maps, which we will explain further in this section.

In Section \ref{sec4}, we will apply our results from the previous section to certain classes of Legendre formulas obtained from the Bhargava factorial, including over $\mathbb{P}$ and $S'$ discussed earlier, and compare them to their respective asymptotics.

\section{Preliminaries.} \label{sec2}
Before we proceed, we let $\left \lfloor x \right \rfloor$ denote the floor function of a real number $x$, $\log$ denote the natural logarithm, and $\mathbb{N} = \mathbb{Z}_+ = \{1,2,3,\dots\}$ denote the natural numbers.
\begin{definition} \label{def4}
   Let $f: \mathbb{R} \to \mathbb{R}$ such that $f$ is positive over $\mathbb{P}$. The \textit{associated Legendre formula for} $f$, or the $f$-\textit{factorial}, is defined as 
    $$n!_{f(x)} = \prod_{p \in \mathbb{P}} p^{\sum_{k\geq 0} \left\lfloor \frac{n}{f(p)p^k} \right\rfloor} .$$
For short, we may write $n!_f$ instead.
\end{definition}
\begin{example} \label{ex5}
Let $f(x) = x$ be the identity map; by Legendre's formula, we have
$$n!_x = n! = \prod_{p \in \mathbb{P}} p^{\sum_{k\geq 0} \left\lfloor \frac{n}{p^{k+1}} \right\rfloor},$$
with the sequence $\{n!_x\}_{n\geq 0} = \{1, 1, 2, 6, 24, 120, \dots \}.$
\end{example}
\begin{example} \label{ex6}
Let $f(x) = \log x $; we have
$$n!_{\log x} = \prod_{p \in \mathbb{P}} p^{\sum_{k\geq 0} \left\lfloor \frac{n}{p^{k}\log p} \right\rfloor},$$
with the sequence $\{n!_{\log x}\}_{n\geq 0} = \{1, 2, 840,  1862340480, \dots \}.$
\end{example}
While the above examples produce interesting and known examples of Legendre formulas, there are other examples for which $n!_f$ is not defined.
\begin{example} \label{ex7}
Let $f(x) = |\sin x| $; we have
$$n!_{|\sin x|} = \prod_{p \in \mathbb{P}} p^{\sum_{k\geq 0} \left\lfloor \frac{n}{p^{k}|\sin p|} \right\rfloor}.$$
However, since $|\csc x| \geq 1$, we have
$$n!_{|\sin x|} \geq  \prod_{p \in \mathbb{P}} p^{n},$$
where the product diverges for $n \geq 1$. Thus, $n!_{|\sin x|}$ is not defined for all $n$.
\end{example}

Based on this, we find, whenever $f$ tends to infinity along $\mathbb{P}$, that its $f$-factorial $n!_f$ is a natural number. In particular, one of the interesting facts about $n!_f$ is its being a natural number is equivalent to the first exponential term being a natural number; in addition, $n!_f$ is a natural number whenever we consider the classes of functions $f$ mentioned previously. 
\begin{theorem} \label{thm8}
Let $f: \mathbb{R} \to \mathbb{R}$ be positive over $\mathbb{P}$. The following are equivalent:
\begin{enumerate}
    \item The sum $\displaystyle \sum_{p \in \mathbb{P}} \log p \left\lfloor \frac{n}{f(p)} \right\rfloor$
converges to $\log \,  k$ for some natural number $k$.
    \item $n!_f$ is a natural number.
\end{enumerate}
\end{theorem}To prove Theorem \ref{thm8} we shall first establish the following lemma.
\begin{lemma} \label{lem9}
Let $b_k \geq 0$ denote a strictly decreasing sequence of nonnegative integers; then the sequence defined by
\begin{align}a_k = \prod_{p \in \mathbb{P}} p^{b_k},\end{align}
with $a_0$ assumed to be finite, is a strictly decreasing sequence of natural numbers such that $a_N = a_{N+1} = \dots = 1$ for some $N$.
\end{lemma}
\begin{proof}Since $b_k$ is strictly decreasing, we have $a_k$ is strictly decreasing; furthermore, since $b_k$ is a strictly decreasing sequence of nonnegative integers, by the pigeon-hole principle, there exists an $N$ such $b_N = b_{N+1} = \dots = 0$. Consequently, we have $a_N = a_{N+1} = \dots = 1$, as desired.\end{proof} 

Before we proceed to prove Theorem \ref{thm8}, using Lemma \ref{lem9}, we note that $n!_f$ shares some properties with the Bhargava factorial. Specificially, one of the properties they share is that $n!_f$, like the Bhargava factorial, satisfies a binomial coefficient formula.
\begin{corollary} \label{cor10}
\textit{For arbitrary nonnegative integers} $n,k$, \textit{the} $f$-\textit{factorial} $(n+k)!_f$ \textit{is a multiple of} $n!_f k!_f$.
\end{corollary}

In addition, $n!_f$ and the Bhargava factorial both satisfy a comparison theorem; that is, in \cite[Lemma 13]{bhargava}, M. Bhargava noted that if $T,S$ are two subsets of $\mathbb{Z}$ and $n!_T$, $n!_S$ are their respective Bhargava factorials such that $T \subseteq S$, then $n!_S \, | \, n!_T$ for every $n \geq 0$.
\begin{corollary} \label{cor11}
\textit{Let} $f,g: \mathbb{R} \to \mathbb{R}$ \textit{be real-valued maps and let} $n!_f, n!_g$ \textit{denote their respective} $f$-\textit{factorials.} \textit{Suppose} $f(p) \leq g(p)$ \textit{for all primes} $p$\textit{; then} $n!_g \, | \, n!_f$\textit{.}
\end{corollary}

As an example, if $f(x) = x-1$ and $g(x) = x$, we have $n!_f = (n+1)!_\mathbb{P}$ and $n!_g = n!$, so that $n! \, | \, (n+1)!_\mathbb{P}$. Finally, we note that, since the Bhargava factorial $n!_S$ is a multiple of $n!$ for any set $S \subseteq \mathbb{Z}$, if we restrict the classes of $f$ to functions that satisfy $f(p) \leq p$ for all primes $p$, by Corollary \ref{cor11}, then $n!_f$ is a multiple of $n!$.

We now prove Theorem \ref{thm8}.
\begin{proof}
Suppose that the sum $$\sum_{p \in \mathbb{P}} \log p \left\lfloor \frac{n}{f(p)} \right\rfloor$$
converges to $\log \,  k$ for some natural number $k$. Let $b_k = \left\lfloor n/(f(p)p^k) \right\rfloor$; then $$a_k = \prod_{p \in \mathbb{P}} p^{\left\lfloor \frac{n}{f(p)p^k} \right\rfloor},$$
and $a_0$ is finite by our assumption, so that, by Lemma \ref{lem9}, we have then $n!_f = a_1a_2 \cdots a_{N-1}$, implying the naturality of $n!_f$.

Assume $n!_f$ is a natural number; by the fundamental theorem of arithmetic, $n!_f = p_1^{a_1}p_2^{a_2}\cdots p_r^{a_r}$, so that the $f$-factorial has the closed form
\begin{equation} \label{firstterm}
\sum_{p \in \mathbb{P}} \log p \sum_{k=0}^\infty \left\lfloor \frac{n}{f(p)p^k} \right\rfloor  = \log \prod_{1 \leq i \leq r} p_i^{a_i}. \end{equation}
Since the sum converges, and the right-hand sum in (\ref{firstterm}) is a logarithm of a natural number, all of the terms in the right-hand sum are finite, including the sum
$$\sum_{p \in \mathbb{P}} \log p \left\lfloor \frac{n}{f(p)} \right\rfloor.$$
As a consequence of convergence, the terms must tend to 0; in other words, for any $\epsilon > 0$, there is some natural number, call it $n_f$, such that for all primes $p \geq n_f$, we have $\left\lfloor n/f(p) \right\rfloor < \epsilon/\log p$. We choose $n_f$ large enough such that there is an $\epsilon$ satisfying $\epsilon < \log p$; then $ \left\lfloor n/f(p) \right\rfloor = 0$, so that the sum is
$$\sum_{p \in \mathbb{P}} \log p \left\lfloor \frac{n}{f(p)} \right\rfloor = \log \prod_{p < n_f} p^{\left\lfloor \frac{n}{f(p)} \right\rfloor},$$
which implies (\ref{ffactorial}), as desired.
\end{proof}
As mentioned previously, we have the following relationship between the classes of $f$ that tend towards infinity and their $f$-factorials.
\begin{theorem}
If $f$ tends to infinity along $\mathbb{P}$, then $n!_f$ is a natural number for all $n$.
\end{theorem}
\begin{proof}
Since $f$ tends to infinity along $\mathbb{P}$, for every $n$, there exists a prime $p'$ such that, for primes $p \geq p'$, $f(p) > n$. Partition $\mathbb{P} = \mathbb{P}_f \cup \mathbb{P}_f'$, where
$$\mathbb{P}_f = \{p \in \mathbb{P} \, | \, f(p) \leq n\},$$
and $\mathbb{P}_f'$ is its complement. Observe that $\mathbb{P}_f$ is finite, i.e., $|\mathbb{P}_f| \leq p'$; then the sum can be decomposed as
\begin{equation} \label{eqn4}
\sum_{p \in \mathbb{P}} \log p \left\lfloor \frac{n}{f(p)} \right\rfloor = \log \prod_{p \in \mathbb{P}_f}  p^{\left\lfloor \frac{n}{f(p)} \right\rfloor}.\end{equation}
Since the term on the right in (\ref{eqn4}) is a natural number, by Theorem \ref{thm8}, we have $n!_f$ is a natural number, as desired.
\end{proof}

From this point forward, we will assume $n!_f$ is a natural number. We will conclude this section by observing, for real-valued functions $f$ satisfying $0 < f(p) \leq p$ for prime $p$, like the Bhargava factorial, is an example of an abstract factorial \cite{minga}.

\begin{definition}
An abstract (or generalized) factorial is a function $!_a: \mathbb{N} \to \mathbb{Z}_+$ that satisfies the following conditions:
\begin{enumerate}
\item $0!_a = 1$.
\item For every nonnegative integers $n$, $k$ such that $0 \leq k \leq n$, the generalized binomial coefficients $$\binom{n}{k}_a = \frac{n!_a}{(n-k)!_ak!_a} \in \mathbb{Z}_+.$$
\item For every natural number $n$, $n!$ divides $n!_a$.
\end{enumerate}
For $n!_f$, this follows from the fact, by construction, $0!_f = 1$ and the second and third conditions follow from Corollary 10 and 11, respectively. As a consequence of being an abstract factorial, $n!_f$ admits various interesting properties, including that the sum
$$\sum_{n=0}^\infty \frac{1}{n!_f}$$
is irrational and that no integer $k \geq 2$ satisfies $k!_f = (k+1)!_f = (k+2)!_f$. In addition, open questions about abstract factorials extend to $n!_f$, including whether there exists a function $\phi$ of bounded variation on $[0,\infty)$ such that
$$n!_f = \int_0^\infty x^n d\phi(x).$$
\end{definition}

\section{Asymptotic Formula.} \label{sec3}
The question of when $n!_f$ admits an asymptotic expansion for arbitrary $f$ seems out of reach due to the various classes of functions $f$ that we would need to consider; as an example, if $f(x) = \sqrt{x}$, its $f$-factorial would not admit an asymptotic expansion via Theorem \ref{thm13} since it does not satisfy the criteria described in its premise and thus we would need a stronger method of approximation. For our purposes, we will consider maps that can be approximated ``well enough" by linear maps in the following sense:
\begin{theorem} \label{thm13}
Let $f$ be real-valued and positive over $\mathbb{P}$. Suppose there exists a positive integer $\alpha$ and a real constant $ M \geq 0$ such that
$$0 \leq \frac{1}{f(p)} -\frac{\alpha}{p} \leq \frac{M}{p^2}.$$
Then
$$\log n!_f = \log (\alpha n)!+ \beta_fn +o(n),$$
for some nonnegative constant $\beta_f$ depending on $f$.
\end{theorem}

In proving Theorem \ref{thm13}, we establish the following lemma.

\begin{lemma} \label{lem14}
For any natural number $n$, we have
\begin{align*}
\sum_{p \in \mathbb{P}} \log p\sum_{k \geq 0}\left\lfloor\frac{\alpha n}{p^{k+1}} + \frac{Mn}{p^{k+2}} \right\rfloor = \sum_{p \in \mathbb{P}} \log p \left\lfloor\frac{n}{p-1}\left(\alpha + \frac{M}{p}\right) \right\rfloor + o(n).
\end{align*}
\end{lemma}
We will first define the first $f$-Chebyshev function and second $f$-Chebyshev function.
\begin{definition} \label{def15}
 Let $f$ be a real-valued map such that $f$ tends to infinity and is positive over $\mathbb{P}$. Then the \emph{first} $f$-\emph{Chebyshev function} and \emph{second} $f$-\emph{Chebyshev function} are defined, respectively, as
\begin{align*}
    \vartheta_f(x) &= \sum_{f(p) \leq x} \log p \quad \text{and} \quad  \psi_f(x) = \sum_{k\geq 0}\vartheta_{g_{k}}(x),
\end{align*}
where $g_k := g_k(p) = f(p)p^k.$
\end{definition}
Observe that, since $f$ is positive and tends to infinity over $\mathbb{P}$, for sufficiently small $x$, both $\vartheta_f(x)$ and $\psi_f(x)$ are zero. Furthermore, we have the following corollary.
\begin{corollary} \label{cor16}
\textit{Let} $f$ \textit{be strictly increasing and bijective over} $\mathbb{P}$\textit{; then} $\vartheta_f = \vartheta \circ f^{-1}$\textit{.}
\end{corollary}
We will now prove Lemma \ref{lem14}.
\begin{proof}
Consider the sum
$$\sum_{p \in \mathbb{P}} \log p\sum_{k \geq 0}\left\lfloor\frac{\alpha n}{p^{k+1}} + \frac{Mn}{p^{k+2}} \right\rfloor.$$
We have
\begin{align*}
\sum_{p \in \mathbb{P}} \log p \sum_{k \geq 0} \left\lfloor\frac{\alpha n}{p^{k+1}} + \frac{Mn}{p^{k+2}} \right\rfloor &= \sum_{m \geq 1, k \geq 0} \sum_{m \leq \frac{n(\alpha p+ M)}{p^{k+2}}  < m+1} \log p \\
&=\sum_{m \geq 1, k \geq 0} m \left(\vartheta_{g_k}\left(\frac{\alpha n}{m}\right) - \vartheta_{g_k}\left(\frac{\alpha n}{m+1}\right)\right) \\
&= \sum_{m \geq 1}m\left(\psi_{f}\left(\frac{\alpha n}{m}\right) - \psi_{f}\left(\frac{\alpha n}{m+1}\right)\right) \\
&= -\lim_{l \to \infty} l\psi_f\left(\frac{\alpha n}{l+1}\right) + \sum_{m \geq 1} \psi_f\left(\frac{\alpha n}{m}\right).
\end{align*}
For large enough $l$, we have $\psi_f\left(\alpha n/l\right) = 0$, so that
$$\sum_{p \in \mathbb{P}} \log p \sum_{k \geq 0} \left\lfloor\frac{\alpha n}{p^{k+1}} + \frac{Mn}{p^{k+2}} \right\rfloor = \sum_{m \geq 1} \psi_f\left(\frac{\alpha n}{m}\right),$$
where $f(x) = \alpha x^2/(\alpha x + M)$. Similarly, for
$$\sum_{p \in \mathbb{P}} \log p \left\lfloor\frac{n}{p-1}\left(\alpha + \frac{M}{p}\right) \right\rfloor,$$
we have
\begin{align*}
\sum_{p \in \mathbb{P}} \log p \left\lfloor\frac{n}{p-1}\left(\alpha + \frac{M}{p}\right) \right\rfloor &= \sum_{m \geq 1} m \sum_{m \leq \frac{n(\alpha p+ M)}{p(p-1)}  < m+1} \log p \\
&=\sum_{m \geq 1} m \left(\vartheta_{h}\left(\frac{\alpha n}{m}\right) - \vartheta_{h}\left(\frac{\alpha n}{m+1}\right)\right) \\
&=\sum_{m \geq 1} \vartheta_{h}\left(\frac{\alpha n}{m}\right)
\end{align*}
where $h(x) = \alpha x(x-1)/(\alpha x + M).$ Thus, establishing Lemma \ref{lem14} is equivalent to showing
$$\lim_{n \to \infty} \frac{1}{n} \sum_{m \geq 1} \psi_f\left(\frac{\alpha n}{m}\right) - \vartheta_{h}\left(\frac{\alpha n}{m}\right) = 0.$$
Since $\psi_f\left(\alpha n/m\right)$ and $\vartheta_{h}\left(\alpha n/m\right)$ are zero for $m >\left(\alpha+ \frac{M}{2}\right)n$, we have that the sum converges to the improper Riemann integral
\begin{align*}
\lim_{n \to \infty} &\frac{1}{n} \sum_{m \leq (\alpha+ \frac{M}{2})n } \psi_f\left(\frac{\alpha n}{m}\right) - \vartheta_{h}\left(\frac{\alpha n}{m}\right) = \frac{\alpha}{q}\int_0^1 \psi_f\left(\frac{q}{x}\right) - \vartheta_{h}\left(\frac{q}{x}\right) \, dx.
\end{align*}
Letting $u = q/x$, where $q = 2\alpha/(2\alpha+M)$, we have
\begin{align*}
\frac{\alpha}{q}\int_0^1 \psi_f\left(\frac{q}{x}\right) - \vartheta_{h}\left(\frac{q}{x}\right) \, dx = \alpha \int_{q}^\infty \frac{\psi_f(u) - \vartheta_{h}(u)}{u^2} \, du.
\end{align*}
We first rewrite the integral as
\begin{align*}
& \int_{q}^\infty \frac{\psi_f(u) - \vartheta_{h}(u)}{u^2} \, du =  \int_{q}^\infty \frac{\psi_f(u) - \vartheta_{f}(u)}{u^2} \, du - \int_{q}^\infty \frac{\vartheta_{h}(u)-\vartheta_f(u)  }{u^2} \, du,
\end{align*}
so that we may focus on each of the respective integrals
\begin{align}
&\alpha \int_{q}^\infty \frac{\psi_f(u) - \vartheta_{f}(u)}{u^2} \, du, \label{eqn5} \\
&\alpha \int_{q}^\infty \frac{\vartheta_{h}(u) - \vartheta_{f}(u)}{u^2} \, du. \label{eqn6}
\end{align}
Each of these improper integrals is convergent; to see this, let's first consider (\ref{eqn5}). From the definition of $\vartheta_f(x)$, we have
$$\vartheta_f(x) \leq \sum_{p \leq \alpha x} \log \alpha x \leq \alpha x \log \alpha x,$$
so that
\begin{align*}
0 \leq \frac{\psi_f(x) - \vartheta_f(x)}{x} &\leq \frac{1}{x}\sum_{2 \leq k \leq \log_2 \alpha x} (\alpha x)^{1/k} \log (\alpha x)^{1/k} \leq  \frac{\sqrt{\alpha} \log^2 \alpha x}{2  \log 2 \sqrt{x}},
\end{align*}
which, by the monotone convergence theorem, implies the convergence of (\ref{eqn5}). As for (\ref{eqn6}), observe that
\begin{equation} \label{eqn7}
h^{-1}(x) = \frac{x+1}{2}+ \sqrt{\frac{(x+1)^2}{4}+\frac{xM}{\alpha}} \quad \text{and} \quad f^{-1}(x) = \frac{x}{2} + \sqrt{\frac{x^2}{4} + \frac{xM}{\alpha}}.
\end{equation}
Since $f \geq h$, we have $\vartheta_{h} \geq \vartheta_f$; using Corollary \ref{cor16}, this implies
\begin{align*}
(\vartheta_h - \vartheta_f)(x) = \sum_{f^{-1}(x) < p \leq h^{-1}(x)} \log p &\leq \log \left(x+\sqrt{\frac{Mx}{\alpha}}+1 \right) \sum_{f^{-1}(x) < p \leq h^{-1}(x)} 1 \\
&\leq \left(\frac{3}{2}+\sqrt{x}\right) \log \left(x+\sqrt{\frac{Mx}{\alpha}}+1 \right),
\end{align*}
which, again, by the monotone convergence theorem, implies the convergence of (6). Observe that we may extend the lower limit of integration in both (\ref{eqn5}) and (\ref{eqn6}) to $0$ since the integrands are zero for sufficiently small $u$. Consider (\ref{eqn5}); recalling $f(x) = \alpha x^{2}/(\alpha x+M)$ and applying Corollary \ref{cor16}, we have the representation
\begin{align*}
\alpha  \int_0^\infty \frac{\psi_f(u) - \vartheta_f(u)}{u^2} \, du  &= \alpha \int_{0}^\infty \frac{1}{u^2} \sum_{k \geq 1} \vartheta_{g_k}(u) \, du \\
&= \alpha \sum_{k \geq 1} \int_0^\infty \frac{\vartheta_{g_k}(u)}{u^2} \, du \\
&=\alpha \sum_{k \geq 1} \int_0^\infty \vartheta(u) \frac{f'(u)u^k+kf(u)u^{k-1}}{f^2(u)u^{2k}}\, du \\
&=\alpha \int_0^\infty \vartheta(u) \left(\frac{f'(u)(u-1)+f(u)}{f^2(u)(u-1)^2}\right)\, du \\
&=\int_0^\infty \frac{\vartheta(u)}{u^3} \left(\frac{2\alpha u^2+(3M-\alpha)u-2M}{(u-1)^2}\right)\, du.
\end{align*}

The justification for the interchange between the summation and integral signs follows from Tonelli's theorem since $\vartheta_{g_k}(u)/u^2$ is nonnegative for every positive real $u$ and natural number $k$. Now consider (\ref{eqn6}); let $n$ be a natural number and consider
$$\alpha \int_{0}^n \frac{\vartheta_{h}(u) - \vartheta_{f}(u)}{u^2} \, du.$$
Recall (\ref{eqn7}); since $h^{-1} \geq f^{-1}$, applying Corollary \ref{cor16}, the following difference has the representation
\begin{align*}
\alpha \int_{0}^n & \frac{\vartheta_{h}(u) - \vartheta_{f}(u)}{u^2} \, du - \alpha \int_{f^{-1}(n)}^{h^{-1}(n)} \vartheta(u)\frac{f'(u)}{f^2(u)} \, du\\ &= \alpha \int_{0}^{h^{-1}(n)} \vartheta(u)\left(\frac{h'(u)}{h^2(u)} - \frac{f'(u)}{f^2(u)} \right) \, du \\
&= \int_{0}^{h^{-1}(n)} \frac{\vartheta(u)}{u^3}\left(\frac{2\alpha u^2+(3M-\alpha)u-2M}{(u-1)^2}\right) \, du.
\end{align*}
It remains to show that
\begin{equation} \label{eqn8}
\alpha \int_{f^{-1}(n)}^{h^{-1}(n)} \vartheta(u)\frac{f'(u)}{f^2(u)} \, du \to 0,
\end{equation}
for arbitrarily large $n$. In \cite{dusart}, P. Dusart showed that $\vartheta(x) < Cx$ for some $C > 1$; thus we have
 \begin{align*}
\alpha \int_{f^{-1}(n)}^{h^{-1}(n)} \vartheta(u)\frac{f'(u)}{f^2(u)} \, du &< C \int_{f^{-1}(n)}^{h^{-1}(n)} \frac{\alpha u^2+2Mu}{u^3} \, du \leq  C\alpha \ln \frac{h^{-1}(n)}{f^{-1}(n)} + \frac{K}{n},
\end{align*}
for some constant $K > 0$. Then, for large $n$, (\ref{eqn8}) follows. As a consequence, the integrals in (\ref{eqn5}) and (\ref{eqn6}) cancel each other, effectively proving Lemma \ref{lem14}.
\end{proof}

Now that we have established Lemma \ref{lem14}, we now prove Theorem \ref{thm13}.

\begin{proof}
Let $\mu$ denote the counting measure; define a sequence of functions $\{h_{n}(p,k)\}_{n\geq 1}$ in two variables as
$$h_{n}(p,k) = \frac{\log p}{n} \left( \left\lfloor \frac{n}{f(p)p^k} \right\rfloor - \left\lfloor \frac{\alpha n}{p^{k+1}} \right\rfloor \right).$$
Since $\lfloor nx\rfloor /n \to x$ pointwise for any real $x$, it follows that $h_{n}(p,k) \to h(p,k)$ pointwise, where
$$h(p,k) = \frac{\log p}{p^k}  \left( \frac{1}{f(p)}  -  \frac{\alpha}{p} \right).$$
We are interested in passing the limit under the integrals
$$\lim_{n \to \infty} \int_\mathbb{P} \int_{\mathbb{Z}_{\geq 0}} h_{n}(p,k)\, d\mu(k)\, d\mu(p) = \int_\mathbb{P} \int_{\mathbb{Z}_{\geq 0}} h(p,k)\, d\mu(k)\, d\mu(p),$$
which will require a doubly application of the dominated convergence theorem. By our assumption, we have the series of inequalities
\begin{equation}\label{upperbound1}
    h_{n}(p,k) \leq \frac{\log p}{n} \left(\left\lfloor\frac{\alpha n}{p^{k+1}} + \frac{Mn}{p^{k+2}} \right\rfloor  - \left\lfloor \frac{\alpha n}{p^{k+1}} \right\rfloor\right) \leq \frac{\log p}{p^{k}}\left(\frac{\alpha n}{p} + \frac{Mn}{p^2}\right).
\end{equation}
The upper bound in (\ref{upperbound1}) is integrable; in particular, we have
\begin{align*}
\int_{\mathbb{Z}_{\geq 0}}  \frac{\log p}{p^k} \left(\frac{\alpha n}{p} + \frac{Mn}{p^2}\right) \, d\mu(k) &= 
\sum_{k \geq 0} \frac{\log p}{p^k} \left(\frac{\alpha n}{p} + \frac{Mn}{p^2}\right)\\
&= \frac{p\log p}{p-1} \left(\frac{\alpha n}{p} + \frac{Mn}{p^2}\right),
\end{align*}
so that, by the dominated convergence theorem, we can pass the limit under the inner integral. For the outer integral, we have
\begin{equation} \label{eqn10}
\int_{\mathbb{Z}_{\geq 0}} h_{n}(p,k)\, d\mu(k) = \frac{\log p}{n} \sum_{k \geq 0}  \left( \left\lfloor \frac{n}{f(p)p^k} \right\rfloor  - \left\lfloor \frac{\alpha n}{p^{k+1}} \right\rfloor  \right).
\end{equation}
Observe that (\ref{eqn10}) converges pointwise to $\displaystyle \int_{\mathbb{Z}_{\geq 0}} h(p,k) d\mu(k)$ since we have the bound
\begin{align*}
\left|\int_{\mathbb{Z}_{\geq 0}} h_{n}(p,k)\, d\mu(k) - \sum_{k \leq a_n}    \log p \left(\frac{1}{f(p)p^k}  -  \frac{\alpha}{p^{k+1}}  \right)\right|  \leq \frac{a_n\log p}{n},
\end{align*}
where $a_n = \log_p \mathrm{max}\{n/f(p), \alpha n/p\}$. By the squeeze theorem, we have
\begin{equation} \label{eqn11}
\int_{\mathbb{Z}_{\geq 0}} h_{n}(p,k)\, d\mu(k) \to\int_{\mathbb{Z}_{\geq 0}} h(p,k)\, d\mu(k) = \frac{p \log p}{p-1}\left(\frac{1}{f(p)} - \frac{\alpha}{p}\right).
\end{equation}
Define
$$g_n(p) = \frac{\log p}{n}\sum_{k \geq 0} \left\lfloor\frac{\alpha n}{p^{k+1}} + \frac{Mn}{p^{k+2}} \right\rfloor - \left\lfloor \frac{\alpha n}{p^{k+1}} \right\rfloor.$$
By our assumptions, $g_n(p)$ is an upper bound for (\ref{eqn10}) and converges pointwise to
$$g(p) = \frac{M\log p}{p(p-1)}. $$
To see this, notice the sum in $g_n(p)$ is nonzero for $k \leq b_n := \log_p (\mathrm{max}\{\alpha, M\} n)$; then, similar to (\ref{eqn10}), we have the inequality
\begin{equation} \label{eqn12}
\left|g_n(p) - \log p \sum_{k \leq b_n} \frac{M}{p^{k+2}} \right| \leq  \frac{b_n\log p}{n},
\end{equation}
so that, by the squeeze theorem, $g_n(p) \to g(p)$. We require that this limit also gives
$$\lim_{n \to \infty} \int_{\mathbb{P}} g_n(p) d\mu(p) = \int_{\mathbb{P}} g(p) d\mu(p).$$By Lemma \ref{lem14}, we have
\begin{align*}
\lim_{n \to \infty} \int_\mathbb{P} g_n(p) \, d\mu(p) 
&= \lim_{n \to \infty} \frac{1}{n} \sum_{p \in \mathbb{P}} \log p \sum_{k \geq 0}\left( \left\lfloor \frac{\alpha n}{p^{k+1}}+\frac{Mn}{p^{k+2}}\right\rfloor - \left\lfloor \frac{\alpha n}{p^{k+1}}\right\rfloor \right) \\
&= \lim_{n \to \infty} \frac{1}{n} \sum_{p \in \mathbb{P}} \log p \left( \left\lfloor \frac{n}{p-1}\left(\alpha+\frac{M}{p}\right)\right\rfloor - \left\lfloor \frac{\alpha  n}{p-1}\right\rfloor \right).
\end{align*}
We want to show that, for any natural number $n$, the following holds up to $o(n)$:
\begin{equation} \label{eqn13}
\sum_{p \in \mathbb{P}} \log p \left(\left\lfloor \frac{n}{p-1}\left(\alpha+\frac{M}{p}\right) \right\rfloor - \left\lfloor \frac{n}{p-1} \right\rfloor\right)\sim  \sum_{p \in \mathbb{P}} \log p \left\lfloor\frac{Mn}{p(p-1)} \right\rfloor.
\end{equation}
Fortunately, the approach to establishing this is akin to the proof of Lemma \ref{lem14}; observe that if 
\begin{equation*}
h = \frac{\alpha x(x-1)}{\alpha x + M}, \quad f_1 = \alpha(x-1), \, \,  \text{ and } \, \, f_2 = \frac{\alpha x(x-1)}{M},
\end{equation*}
the left-hand sum in (\ref{eqn13}) may be written as
\begin{equation*}
\lim_{n\to \infty} \frac{1}{n} \sum_{m \leq (\alpha+M/2)n} \left[\vartheta_{h}\left(\frac{\alpha n}{m}\right) -  \vartheta_{f_1}\left(\frac{\alpha n}{m}\right) - \vartheta_{f_2}\left(\frac{\alpha n}{m}\right)\right],
\end{equation*}
where each of the summands is zero for $m > (\alpha+M/2)n$. Taking the limit, we have the Riemann sum converges to the improper integral
\begin{equation*}
 \frac{\alpha}{q} \int_0^1 \left(\vartheta_h - \vartheta_{f_1}-\vartheta_{f_2}\right)\left(\frac{q}{x}\right) \, dx = \alpha \int_q^\infty \frac{\left(\vartheta_h - \vartheta_{f_1}- \vartheta_{f_2}\right)\left(u\right)}{u^2}\, du,
\end{equation*}
where $q = 2\alpha/(2\alpha + M)$ and the same $u$-substitution was applied in the proof of Lemma \ref{lem14}. Again, we may extend the lower limit of integration to $0$ since the integrand is zero for sufficiently small $u$. Consider
$$\alpha \int_0^n \frac{\vartheta_h\left(u\right) - \vartheta_{f_1}\left(u\right) - \vartheta_{f_2}\left(u\right)}{u^2}\, du.$$
Applying Corollary \ref{cor16}, we have
\begin{align*}
&\alpha \int_0^n \frac{\vartheta_h\left(u\right) - \vartheta_{f_1}\left(u\right) - \vartheta_{f_2}\left(u\right)}{u^2}\, du \\
&= \alpha \int_0^n \frac{\vartheta_h\left(u\right) - \vartheta_{f_1}\left(u\right)}{u^2} \, du - \alpha \int _0^n \frac{\vartheta_{f_2}\left(u\right)}{u^2}\, du \\
&= \alpha \int_0^{h^{-1}(n)} \vartheta(u) \left(\frac{h'(u)}{h^2(u)} - \frac{f_1'(u)}{f_1^2(u)}\right) \, du \\
&\qquad -\alpha \int_0^{f_2^{-1}(n)} \vartheta(u) \frac{f_2'(u)}{f_2^2(u)} \, du + \alpha \int_{f_1^{-1}(n)}^{h^{-1}(n)} \vartheta(u) \frac{f_1'(u)}{f_1^2(u)} \, du \\
&= M \left(\int_0^{h^{-1}(n)} - \int_0^{f_2^{-1}(n)}\right)\left( \frac{\vartheta(u)}{u^2} \frac{2u-1}{(u-1)^2} \, du  \right) + \int_{f_1^{-1}(n)}^{h^{-1}(n)} \frac{\vartheta(u)}{(u-1)^2} \, du.
\end{align*}
Again, similar to the proof of Lemma \ref{lem14}, we have
$$\int_{f_1^{-1}(n)}^{h^{-1}(n)} \frac{\vartheta(u)}{(u-1)^2} \, du \to 0,$$
so that we have
$$M \left(\int_0^{h^{-1}(n)} - \int_0^{f_2^{-1}(n)}\right)\left( \frac{\vartheta(u)}{u^2} \frac{2u-1}{(u-1)^2} \, du  \right)  \to 0,$$
implying (\ref{eqn12}).

Putting (\ref{eqn11}) and (\ref{eqn13}) together, by the dominated convergence theorem, we have
$$\lim_{n \to \infty} \int_\mathbb{P} \int_{\mathbb{Z}_{\geq 0}} h_{n}(p,k)\, d\mu(k)\, d\mu(p) = \int_\mathbb{P} \int_{\mathbb{Z}_{\geq 0}} h(p,k)\, d\mu(k)\, d\mu(p),$$
where
\begin{equation}\label{beta}\int_\mathbb{P} \int_{\mathbb{Z}_{\geq 0}} h(p,k)\, d\mu(k)\, d\mu(p) = \sum_{p \in \mathbb{P}}  \frac{p\log p}{p-1} \left( \frac{1}{f(p)} - \frac{\alpha}{p} \right).
\end{equation}
Denote (\ref{beta}) by $\beta_f$; from our assumptions, we have $\beta_f \geq 0$. Furthermore, we have
\begin{equation*}
    \sum_{p \in \mathbb{P}}  \frac{p\log p}{p-1} \left( \frac{1}{f(p)} - \frac{\alpha}{p} \right) \leq \sum_{p \in \mathbb{P}}\frac{M\log p}{p(p-1)} \leq 2M \sum_{p \in \mathbb{P}}\frac{\log p}{p^2} \leq -2M\zeta'(2),
    \end{equation*}
where $\zeta(s)$ is the classical Riemann zeta-function. By the monotone convergence theorem, $\beta_f$ is finite, which proves Theorem \ref{thm13}.
\end{proof}

\section{Applications.} \label{sec4}

We are interested in discussing applications of Theorem \ref{thm13} to various Legendre formulas, in particular focusing on the Bhargava factorials over the set of primes and the set associated with sequence A202367.

\subsection{Bhargava factorial over the set of primes.} 
In this section, we prove Theorem \ref{thm2}. Recall that the Bhargava factorial is given in Definition \ref{def1}.

\begin{proof}[Proof of Theorem \ref{thm2}]
Let $f(x) = x-1$. We find $M=2$ and $\alpha = 1$ satisfy the criteria of Theorem \ref{thm13}, where $n!_f = (n+1)!_\mathbb{P}$. We have
$$\log (n+1)!_\mathbb{P} = \log n!+ Cn + o(n),$$
as desired.
\end{proof}
\begin{table}[h!]
    \centering
    \caption{Table of values comparing $\log (n+1)!_\mathbb{P}$ and $\log n! + {Cn}$.}
    \label{tab:my_label}
    \begin{tabular}{ |p{1cm}||p{4cm}|p{4cm}|  }
    \hline
 $n$ & $\log (n+1)!_\mathbb{P}$ & $\log n!+ Cn$ \\
 \hline
 1   & .6931\ldots             & 1.2269\ldots \\
 2   & 3.1780\ldots            & 3.1470\ldots \\
 3   & 3.8712\ldots             & 5.4726\ldots \\
 4   & 8.6586\ldots          & 8.0859\ldots \\
 5   & 9.3518\ldots         & 10.9223\ldots \\
 6   & 14.8812\ldots       & 13.9410\ldots \\
 7   & 15.5744\ldots       & 17.1139\ldots \\
 8   & 21.0550\ldots    & 20.4203\ldots \\
 9   & 21.7482\ldots   & 23.8445\ldots \\
 10  & 26.6310\ldots   & 27.3741\ldots \\
 100   & 471.9704\ldots    & 480.6040\ldots \\
 1,000  & 7,119.5084\ldots   & 7,130.9600\ldots \\
 5,000  & 43,759.7980\ldots  & 43,726.0000\ldots  \\
 10,000 & 94,417.8375\ldots  & 94,378.6000\ldots \\
 \hline
\end{tabular}
\end{table} Interestingly, the constant $C$ occurs in the study of the Carmichael function $\lambda(n)$, which is defined to be the smallest positive integer satisfying $a^{\lambda(n)} \, \equiv \, 1 \; \; (\mathrm{mod} \, n)$ for every integer $a$ between 1 and $n$ coprime to $n$. In particular, for all numbers $m$ but $o(m)$ positive integers such that $n \leq m$, we have
$$\lambda(n) = \frac{n}{(\log n)^{\log \log \log n + A + o(1)}},$$
where $A = C-1$ (see \cite{sandor}). In addition, the constant $C$ also appears in the study of $\Omega(n)$, which is defined to be the number of not necessarily distinct prime factors of a natural number $n$. In particular, the variance of $\Omega(n)$ is given by the asymptotics
$$\mathrm{var}(\Omega(n)) \sim \log \log n + a_1+ \frac{a_2}{\log n} + \cdots$$
where $a_1 \approx .76478\dots$ and $a_2 = \gamma - 1 -2C$, where $\gamma$ is the Euler--Mascheroni constant (see \cite{finch}).

\subsection{LCM of a class of a polynomials.} Consider the least common multiple of the denominators of the coefficients of polynomials $p_m(n)$ defined by the recursion $$p_m(n) = \sum_{i=1}^n i^2p_{(m-1)}(i) $$
for $m\geq 1$, with $p_0(n)=1$; this generates the sequence
\begin{equation}\label{sequence} 1, 6, 360, 45360, 5443200, 359251200,\dots \end{equation} which is sequence A202367 on \url{www.oeis.org}. Vladimir Shevelev and Peter Moses noted that it is conjectured that the sequence in (\ref{sequence}) is given by
\begin{equation} \label{eqn16}
\prod_{p \in \mathbb{P}} p^{\sum_{k \geq 0} \left \lfloor \frac{n-1}{\left \lceil (p-1)/2 \right \rceil p^k} \right \rfloor}. 
\end{equation}
Let $S' = \{2,4,16,22,\dots\}$ denote the set associated with sequence A202367 so that we have its Bhargava factorial $n!_{S'}$, given by the formula in (\ref{eqn16}). We now prove Theorem \ref{thm3}.
\begin{proof}[Proof of Theorem \ref{thm3}]
Let $f(x) = \left \lceil (x-1)/2 \right \rceil$. We find $M=4$ and $\alpha = 2$ satisfy the criteria of Theorem \ref{thm13}, where $n!_f = (n+1)!_{S'}$. We have
$$\log (n+1)!_{S'} = \log (2n)!+ \beta n + o(n),$$
as desired.
\end{proof}
\newpage
\begin{table}[h!]
    \centering
    \caption{Table of values comparing $\log (n+1)!_{S'}$ and $\log (2n)!+\beta n$.}
    \label{tab:my_label2}
    \begin{tabular}{ |p{1cm}||p{4.75cm}|p{4.75cm}|  }
    \hline
 $n$ & $\log (n+1)!_{S'}$ & $\log (2n)!+ \beta n$  \\
 \hline
 1   & 1.7917\ldots           & 1.7607\ldots \\
 2   & 5.8861\ldots       & 5.3133\ldots \\
 3   & 10.7223\ldots           & 9.7821\ldots \\
 4   & 15.5098\ldots          & 14.8751\ldots \\
 5   & 19.6995\ldots        & 20.4426\ldots \\
 6   & 29.4033\ldots     & 26.3931\ldots \\
 7   & 31.1951\ldots     & 32.6647\ldots \\
 8   & 39.5089\ldots   & 39.2130\ldots \\
 9   & 48.3882\ldots    & 46.0042\ldots \\
 10  & 56.4899\ldots  & 53.0120\ldots \\
 100  & 982.0880\ldots   & 969.9960\ldots \\
 1,000  & 14,288.7934\ldots  & 14,274.2000\ldots \\
 5,000  & 87,486.3657\ldots   & 87,447.1000\ldots \\
 10,000  & 188,805.0729\ldots   & 188,752.0000\ldots \\
 \hline
\end{tabular}
\end{table}
\begin{acknowledgment}{Acknowledgments.}
The author is grateful to Ryan Kinser and David Adam for their remarks on the initial
version of the manuscript and to the reviewers for their feedback on this article. In addition, the author is grateful to Holly Swisher and Samadhi Metta-Bexar for their unconditional support.
\end{acknowledgment}

\begin{biog}
\item[Maiyu Diaz] received their MSc. in mathematics from the University of Iowa and their B.A. in mathematics from St. Mary's University in San Antonio, TX. They are currently completing their Ph.D. in mathematics at Texas Christian University.
\begin{affil}
Department of Mathematics, Texas Christian University, Fort Worth TX 76129\\
\end{affil}

\end{biog}

\vfill\eject

\end{document}